\documentclass[12pt,reqno]{article}

\usepackage[usenames]{color}
\usepackage{amssymb}
\usepackage{amsmath}
\usepackage{amsthm}
\usepackage{amsfonts}
\usepackage{amscd}
\usepackage{graphicx}

\usepackage[colorlinks=true,
linkcolor=webgreen,
filecolor=webbrown,
citecolor=webgreen]{hyperref}

\definecolor{webgreen}{rgb}{0,.5,0}
\definecolor{webbrown}{rgb}{.6,0,0}

\usepackage{color}
\usepackage{fullpage}
\usepackage{float}

\usepackage{graphics}
\usepackage{latexsym}

\begin{document}

\theoremstyle{plain}
\newtheorem{theorem}{Theorem}
\newtheorem{corollary}[theorem]{Corollary}
\newtheorem{lemma}{Lemma}
\newtheorem{proposition}{Proposition}
\newtheorem{example}{Examples}
\newtheorem*{remark}{Remark}

\begin{center}
\vskip 1cm{\LARGE\bf 
Three new classes of binomial Fibonacci sums \\
}
\vskip 1cm
{\large

Robert Frontczak \\
Independent Researcher \\
Reutlingen,  Germany \\
\href{mailto:robert.frontczak@web.de}{\tt robert.frontczak@web.de}}
\end{center}

\vskip .2 in

\begin{abstract}
In this paper, we introduce three new classes of binomial sums involving Fibonacci (Lucas) numbers and weighted binomial coefficients.
\end{abstract}

\vskip 0.2cm


\section{Introduction and motivation}

As usual, we will use the notation $F_n$ for the $n$th Fibonacci number and $L_n$ for the $n$th Lucas number, respectively.
Both number sequences are defined, for \text{$n\in\mathbb Z$}, through the same recurrence relation $x_n = x_{n-1} + x_{n-2}, n\ge 2,$
with initial values $F_0=0, F_1=1$, and $L_0=2, L_1=1$, respectively. They possess the explicit formulas (Binet forms)
\begin{equation*}
F_n = \frac{\alpha^n - \beta^n }{\alpha - \beta },\quad L_n = \alpha^n + \beta^n,\quad n\in\mathbb Z,
\end{equation*}
where $\alpha=(1+\sqrt{5})/2$ is the golden section and $\beta=-1/\alpha$. 
For negative subscripts one checks easily that $F_{-n}=(-1)^{n-1}F_n$ and $L_{-n}=(-1)^n L_n$.
For more information about these famous sequences we refer, among others, to the books by Koshy \cite{Koshy} and Vajda \cite{Vajda}.
In addition, one can consult the On-Line Encyclopedia of Integer Sequences
\cite{OEIS} where these sequences are listed under the ids {A000045} and A000032, respectively.

The literature on Fibonacci numbers is very rich. Dozens of articles and problem proposals 
dealing with binomial sum identities involving these sequences exist. 
Classical articles on the topic are \cite{Carlitz,CarFer,Hoggatt,Hoggatt2,Layman,Zeitlin}, among others. 
Newer contributions include \cite{Jen,Long} and recent articles are \cite{Adegoke1,Adegoke2,Adegoke3,Adegoke4,Bai,Kilic1,Kilic2}.

This note is motivated by the problem proposal \cite{Frontczak} where the author asked to prove the identities
$$
\sum_{k=0}^n \binom{n}{k}\frac{F_k+L_k}{k+1}=\frac{F_{2n+1}+L_{2n+1}}{n+1}\quad {\text{\rm and}}\quad 
\sum_{k=0}^n \binom{n}{k}\frac{F_k+L_k}{(k+1)(k+2)}=\frac{F_{2n+2}+L_{2n+2}-2}{(n+1)(n+2)}.
$$
A solution with a slight generalization was provided by Ventas in \cite{Ventas}.
Here, we introduce some generalized variants of this proposal which should be regarded as attractive complements. 
More precisely, we present three presumable new classes of Fibonacci (Lucas) binomial sums possessing the same structure.
Our results follow from three recently published polynomial identities derived by Dattoli et al. \cite{Dattoli}. 
They are given by ($x\in\mathbb{C}$)
\begin{equation}\label{Dat1}
\sum_{k=0}^n \binom {n}{k} (-1)^k \frac{1}{k+1} x^{k+1} (1+x)^{n-k} = \frac{(1+x)^{n+1} - 1}{n+1},
\end{equation}
\begin{equation}\label{Dat2}
\sum_{k=0}^n \binom {n}{k} (-1)^k \frac{1}{k+2} x^{k+2} (1+x)^{n-k} = \frac{(1+x)^{n+2} - (n+2)x - 1}{(n+1)(n+2)},
\end{equation}
and
\begin{equation}\label{Dat3}
\sum_{k=0}^n \binom {n}{k} (-1)^k \frac{1}{(k+1)(k+2)} x^{k+2} (1+x)^{n-k} = \frac{(n+1)x(1+x)^{n+1} - (1+x)^{n+1} + 1}{(n+1)(n+2)}.
\end{equation}

In the course of derivation we will make use of the following known results.

\begin{lemma}\label{lem_help}
For any integer $s$, we have
\begin{equation}
(- 1)^s + \alpha^{2s} = \alpha^s L_s, \qquad \mbox{and}\qquad (- 1)^s + \beta^{2s} = \beta^s L_s.
\end{equation}
\end{lemma}

\begin{lemma}
Let $r$ and $s$ be any integers. Then it holds that \cite{Hoggatt}
\begin{align}
L_{r + s} - L_r \alpha ^s &= - \beta ^r F_s \sqrt 5, \label{hog1} \\ 
L_{r + s} - L_r \beta ^s &= \alpha ^r F_s \sqrt 5, \label{hog2} \\ 
F_{r + s} - F_r \alpha ^s &= \beta ^r F_s, \label{hog3} \\ 
F_{r + s} - F_r \beta ^s &= \alpha ^r F_s \label{hog4}. 
\end{align}
\end{lemma}

\section{First set of results}

\begin{theorem}\label{main_thm1}
If $r$, $s$ and $t$ are any integers and $n$ is a non-negative integer, then
\begin{align}
&\sum_{k = 0}^n \binom{n}{k} \frac{1}{k+1} (-1)^{s(k+1)+t} F_r^{k+1} F_s^{n-k} F_{rn-s(k+1)-rk-t}  \nonumber \\
&\qquad\qquad\qquad\qquad = \frac{1}{n+1}\Big ( (-1)^{t+1} F_s^{n+1} F_{r(n+1)-t} - F_t F_{r+s}^{n+1} \Big )
\end{align}
and 
\begin{align}
&\sum_{k = 0}^n \binom{n}{k} \frac{1}{k+1} (-1)^{s(k+1)+1+t} F_r^{k+1} F_s^{n-k} L_{rn-s(k+1)-rk-t}  \nonumber \\
&\qquad\qquad\qquad\qquad = \frac{1}{n+1}\Big ( (-1)^{t} F_s^{n+1} L_{r(n+1)-t} - L_t F_{r+s}^{n+1} \Big ).
\end{align}
\end{theorem}
\begin{proof}
Set $x=-F_{r}\alpha^s/F_{r+s}$ in~\eqref{Dat1}, use~\eqref{hog3} and multiply through by $\alpha^t$, to obtain
\begin{equation*}
\sum_{k = 0}^n \binom{n}{k} \frac{1}{k+1} (-1)^{r(n-k)+1} F_{r}^{k+1} F_s^{n - k} \alpha^{k(s+r) - rn + s + t}  
= \frac{1}{n+1}\Big ( (-1)^{r(n+1)} F_s^{n+1} \alpha^{-r(n+1)+t} - \alpha^t F_{r+s}^{n+1}\Big ).
\end{equation*}
Similarly, setting $x=-F_{r}\beta^s/F_{r+s}$ in~\eqref{Dat1}, using~\eqref{hog4} and multiplying through by $\beta^t$, yields
\begin{equation*}
\sum_{k = 0}^n \binom{n}{k} \frac{1}{k+1} (-1)^{r(n-k)+1} F_{r}^{k+1} F_s^{n - k} \beta^{k(s+r) - rn + s + t}  
= \frac{1}{n+1}\Big ( (-1)^{r(n+1)} F_s^{n+1} \beta^{-r(n+1)+t} - \beta^t F_{r+s}^{n+1}\Big ).
\end{equation*}
The results follow by combining these identities according to the Binet forms while applying $F_{-n}=(-1)^{n-1}F_n$ and $L_{-n}=(-1)^n L_n$.
\end{proof}

Theorem \ref{main_thm1} contains many interesting identities as special cases which are presented as corollaries.

\begin{corollary}\label{cor1}
\begin{equation}
\sum_{k = 0}^n \binom{n}{k} \frac{1}{k+1} (-1)^{k} F_{n-2k-1+t} = \frac{1}{n+1}\Big ( F_{n+1+t} - F_t \Big )
\end{equation}
and
\begin{equation}
\sum_{k = 0}^n \binom{n}{k} \frac{1}{k+1} (-1)^{k} L_{n-2k-1+t} = \frac{1}{n+1}\Big ( L_{n+1+t} - L_t \Big ).
\end{equation}
\end{corollary}

\begin{corollary}\label{cor2}
\begin{equation}
\sum_{k = 0}^n \binom{n}{k} \frac{1}{k+1} (-1)^{t} F_{n-3k-2-t} = \frac{1}{n+1}\Big ((-1)^{t+1} F_{n+1-t} - F_t 2^{n+1} \Big )
\end{equation}
and
\begin{equation}
\sum_{k = 0}^n \binom{n}{k} \frac{1}{k+1} (-1)^{t+1} L_{n-3k-2-t} = \frac{1}{n+1}\Big ( (-1)^{t} L_{n+1-t} - L_t 2^{n+1} \Big ).
\end{equation}
\end{corollary}

\begin{corollary}\label{cor3}
\begin{equation}
\sum_{k = 0}^n \binom{n}{k} \frac{1}{k+1} (-1)^{k} F_{2n-3k-1+t} = \frac{1}{n+1}\Big (F_{2n+2+t} - F_t 2^{n+1} \Big )
\end{equation}
and
\begin{equation}
\sum_{k = 0}^n \binom{n}{k} \frac{1}{k+1} (-1)^{k} L_{2n-3k-1+t} = \frac{1}{n+1}\Big (L_{2n+2+t} - L_t 2^{n+1} \Big ).
\end{equation}
\end{corollary}

\begin{corollary}\label{cor4}
\begin{equation}
\sum_{k = 0}^n \binom{n}{k} \frac{1}{k+1} (-1)^{t} F_{2n-4k-2-t} = \frac{1}{n+1}\Big ((-1)^{t+1} F_{2n+2-t} - F_t 3^{n+1} \Big )
\end{equation}
and
\begin{equation}
\sum_{k = 0}^n \binom{n}{k} \frac{1}{k+1} (-1)^{t+1} L_{2n-4k-2-t} = \frac{1}{n+1}\Big ((-1)^t L_{2n+2-t} - L_t 3^{n+1} \Big ).
\end{equation}
\end{corollary}

\begin{corollary}\label{cor5}
\begin{equation}
\sum_{k = 0}^n \binom{n}{k} \frac{1}{k+1} (-1)^{k} F_{2n-k+1+t} = \frac{1}{n+1}\Big (F_{2n+2+t} - F_t \Big )
\end{equation}
and
\begin{equation}
\sum_{k = 0}^n \binom{n}{k} \frac{1}{k+1} (-1)^{k} L_{2n-k+1+t} = \frac{1}{n+1}\Big (L_{2n+2+t} - L_t \Big ).
\end{equation}
\end{corollary}

\begin{corollary}\label{cor6}
\begin{equation}
\sum_{k = 0}^n \binom{n}{k} \frac{1}{k+1} (-1)^{n+k+1} 2^{n-k} F_{n+2k+3+t} = \frac{1}{n+1}\Big ((-2)^{n+1} F_{n+1+t} - F_t \Big )
\end{equation}
and
\begin{equation}
\sum_{k = 0}^n \binom{n}{k} \frac{1}{k+1} (-1)^{n+k+1} 2^{n-k} L_{n+2k+3+t} = \frac{1}{n+1}\Big ((-2)^{n+1} L_{n+1+t} - L_t \Big ).
\end{equation}
\end{corollary}

\begin{corollary}\label{cor7}
\begin{equation}
\sum_{k = 0}^n \binom{n}{k} \frac{1}{k+1} (-1)^{k} 2^{n-k} F_{2n+k+3+t} = \frac{1}{n+1}\Big (2^{n+1} F_{2n+2+t} - F_t \Big )
\end{equation}
and
\begin{equation}
\sum_{k = 0}^n \binom{n}{k} \frac{1}{k+1} (-1)^{k} 2^{n-k} L_{2n+k+3+t} = \frac{1}{n+1}\Big (2^{n+1} L_{2n+2+t} - L_t \Big ).
\end{equation}
\end{corollary}

\begin{corollary}\label{cor8}
\begin{equation}
\sum_{k = 0}^n \binom{n}{k} \frac{1}{k+1} (-1)^{k} 3^{n-k} F_{2(n+k+2)+t} = \frac{1}{n+1}\Big (3^{n+1} F_{2n+2+t} - F_t \Big )
\end{equation}
and
\begin{equation}
\sum_{k = 0}^n \binom{n}{k} \frac{1}{k+1} (-1)^{k} 3^{n-k} L_{2(n+k+2)+t} = \frac{1}{n+1}\Big (3^{n+1} L_{2n+2+t} - L_t \Big ).
\end{equation}
\end{corollary}

\begin{theorem}\label{main_thm2}
If $s$ is an even integer and $t$ is any integer, then
\begin{equation}
\sum_{k = 0}^n \binom{n}{k} \frac{1}{k+1} (-1)^{k} L_s^{n-k} F_{s(n+k+2)+t} 
= \frac{1}{n+1}\Big ( L_s^{n+1} F_{s(n+1)+t} - F_t \Big )
\end{equation}
and 
\begin{equation}
\sum_{k = 0}^n \binom{n}{k} \frac{1}{k+1} (-1)^{k} L_s^{n-k} L_{s(n+k+2)+t} 
= \frac{1}{n+1}\Big ( L_s^{n+1} L_{s(n+1)+t} - L_t \Big ).
\end{equation}
\end{theorem}
\begin{proof}
Let $s$ be even. Set $x=\alpha^{2s}$ and $x=\beta^{2s}$, respectively,  in~\eqref{Dat1}, use~\ref{lem_help}.
Multiply through the resulting equations by $\alpha^t$ and $\beta^t$, respectively, and combine according to the Binet forms.
\end{proof}

\begin{remark}
Note that when $s=2$, Theorem \ref{main_thm2} gives again Corollary \ref{cor8}.
\end{remark}

Working with  $x=-F_{r+s}/(\alpha^s F_{r})$ and $x=-F_{r+s}/(\beta^s F_{r})$, and using the same arguments we get the next results.

\begin{theorem}\label{main_thm3}
If $r$, $s$ and $t$ are any integers and $n$ is a non-negative integer, then
\begin{align}
&\sum_{k = 0}^n \binom{n}{k} \frac{1}{k+1} (-1)^{k} F_{r+s}^{k+1} F_s^{n-k} F_{s(k+1)+(r+s)(n-k)-t}  \nonumber \\
&\qquad\qquad\qquad\qquad = \frac{1}{n+1}\Big ( F_s^{n+1} F_{(r+s)(n+1)-t} + (-1)^{(s+1)(n+1)+t} F_t F_{r}^{n+1} \Big )
\end{align}
and 
\begin{align}
&\sum_{k = 0}^n \binom{n}{k} \frac{1}{k+1} (-1)^{k} F_{r+s}^{k+1} F_s^{n-k} L_{s(k+1)+(r+s)(n-k)-t}  \nonumber \\
&\qquad\qquad\qquad\qquad = \frac{1}{n+1}\Big ( F_s^{n+1} L_{(r+s)(n+1)-t} + (-1)^{(s+1)(n+1)+t+1} L_t F_{r}^{n+1} \Big ).
\end{align}
\end{theorem}

\section{Results from identities \eqref{Dat2} and \eqref{Dat3} }

The results for the other two classes of sums are presented without proofs as the ideas are clear.

\begin{theorem}\label{main_thm4}
If $r$, $s$ and $t$ are any integers and $n$ is a non-negative integer, then
\begin{align}
&\sum_{k = 0}^n \binom{n}{k} \frac{1}{k+2} (-1)^{r(n-k)} F_{r}^{k+2} F_s^{n-k} F_{s(k+2)-r(n-k)+t}  \nonumber \\
&\qquad\qquad = \frac{1}{(n+1)(n+2)}\Big ( (-1)^{t+1} F_s^{n+2} F_{r(n+2)-t} - F_t F_{r+s}^{n+2} \Big )
+ \frac{1}{n+1} F_r F_{s+t} F_{r+s}^{n+1}
\end{align}
and 
\begin{align}
&\sum_{k = 0}^n \binom{n}{k} \frac{1}{k+2} (-1)^{r(n-k)} F_{r}^{k+2} F_s^{n-k} L_{s(k+2)-r(n-k)+t}  \nonumber \\
&\qquad\qquad = \frac{1}{(n+1)(n+2)}\Big ( (-1)^t F_s^{n+2} L_{r(n+2)-t} - L_t F_{r+s}^{n+2} \Big )
+ \frac{1}{n+1} F_r L_{s+t} F_{r+s}^{n+1}.
\end{align}
\end{theorem}

\begin{theorem}\label{main_thm5}
If $s$ is an even integer and $t$ is any integer, then
\begin{equation}
\sum_{k = 0}^n \binom{n}{k} \frac{1}{k+2} (-1)^{k} L_s^{n-k} F_{2s(k+2)+s(n-k)+t}
= \frac{1}{(n+1)(n+2)}\Big ( L_s^{n+2} F_{s(n+2)+t} - F_t \Big ) - \frac{1}{n+1} F_{2s+t}
\end{equation}
and 
\begin{equation}
\sum_{k = 0}^n \binom{n}{k} \frac{1}{k+2} (-1)^{k} L_s^{n-k} L_{2s(k+2)+s(n-k)+t} 
= \frac{1}{(n+1)(n+2)}\Big ( L_s^{n+2} L_{s(n+2)+t} - L_t \Big ) - \frac{1}{n+1} L_{2s+t}.
\end{equation}
\end{theorem}

\begin{theorem}\label{main_thm6}
If $r$, $s$ and $t$ are any integers and $n$ is a non-negative integer, then
\begin{align}
&\sum_{k = 0}^n \binom{n}{k} \frac{1}{(k+1)(k+2)} (-1)^{r(n-k)} F_{r}^{k+2} F_s^{n-k} F_{s(k+2)-r(n-k)+t}  \nonumber \\
&\, = -\frac{1}{(n+1)(n+2)}\Big ( (-1)^{t+1} F_s^{n+1} F_{r+s} F_{r(n+1)-t} - F_t F_{r+s}^{n+2} \Big )
+ \frac{1}{n+2} F_{s}^{n+1} F_r (-1)^{s+t} F_{r(n+1)-s-t}
\end{align}
and 
\begin{align}
&\sum_{k = 0}^n \binom{n}{k} \frac{1}{(k+1)(k+2)} (-1)^{r(n-k)} F_{r}^{k+2} F_s^{n-k} L_{s(k+2)-r(n-k)+t}  \nonumber \\
&\, = -\frac{1}{(n+1)(n+2)}\Big ( (-1)^{t} F_s^{n+1} F_{r+s} L_{r(n+1)-t} - L_t F_{r+s}^{n+2} \Big )
+ \frac{1}{n+2} F_{s}^{n+1} F_r (-1)^{s+t+1} L_{r(n+1)-s-t}.
\end{align}
\end{theorem}

\begin{theorem}\label{main_thm7}
If $s$ is an even integer and $t$ is any integer, then
\begin{align}
&\sum_{k = 0}^n \binom{n}{k} \frac{1}{(k+1)(k+2)} (-1)^{k} L_s^{n-k} F_{s(n+k+4)+t} \nonumber \\
&\qquad\qquad = -\frac{1}{(n+1)(n+2)}\Big ( L_s^{n+1} F_{s(n+1)+t} - F_t \Big ) + \frac{1}{n+2} L_s^{n+1} F_{s(n+3)+t}
\end{align}
and 
\begin{align}
&\sum_{k = 0}^n \binom{n}{k} \frac{1}{(k+1)(k+2)} (-1)^{k} L_s^{n-k} L_{s(n+k+4)+t} \nonumber \\
&\qquad\qquad = -\frac{1}{(n+1)(n+2)}\Big ( L_s^{n+1} L_{s(n+1)+t} - L_t \Big ) + \frac{1}{n+2} L_s^{n+1} L_{s(n+3)+t}.
\end{align}
\end{theorem}

\section{Additional sum relations}

In \cite{Dattoli} the following sum relation is also proved:
\begin{equation}
\sum_{k=0}^n \binom {n}{k} (-1)^k \frac{1}{k+2} x^{k} (1+x)^{n-k} = \sum_{k=0}^n \binom {n}{k} \frac{x^{k}}{(k+1)(k+2)}.
\end{equation}
This relation immediately yields
\begin{equation*}
\sum_{k=0}^n \binom {n}{k} (-1)^k \frac{1}{k+2} F_{2n-k} = \sum_{k=0}^n \binom {n}{k} \frac{F_{k}}{(k+1)(k+2)},
\end{equation*}
\begin{equation*}
\sum_{k=0}^n \binom {n}{k} (-1)^k \frac{1}{k+2} L_{2n-k} = \sum_{k=0}^n \binom {n}{k} \frac{L_{k}}{(k+1)(k+2)}
\end{equation*}
or
\begin{equation*}
\sum_{k=0}^n \binom {n}{k} (-1)^k \frac{F_{2n-k} + L_{2n-k}}{k+2} = \sum_{k=0}^n \binom {n}{k} \frac{F_k+L_k}{(k+1)(k+2)}
= \frac{F_{2n+2} + L_{2n+2} - 2}{(n+1)(n+2)},
\end{equation*}
which provides a nice addendum to problem proposal \cite{Frontczak}. More generally, we have sum relations of the following form.

\begin{theorem}\label{thm_sumrel}
If $r$, $s$ and $t$ are any integers ($r$ non-zero) and $n$ is a non-negative integer, then
\begin{align}
&F_{r+s}^{-n} \sum_{k = 0}^n \binom{n}{k} \frac{1}{k+2} (-1)^{r(n-k)} F_{r}^{k} F_s^{n-k} F_{sk-r(n-k)+t}  \nonumber \\
&= \sum_{k=0}^n \binom{n}{k} \frac{1}{(k+1)(k+2)}(-1)^{k} \Big (\frac{F_r}{F_{r+s}}\Big )^{k} F_{sk+t} \nonumber \\
&=\frac{1}{(n+1)(n+2)}\Big ( \Big (\frac{F_s}{F_r}\Big )^2 \Big (\frac{F_s}{F_{r+s}}\Big )^n (-1)^{t+1} F_{2s+r(n+2)-t}
- \Big (\frac{F_{r+s}}{F_r}\Big )^2 (-1)^{t+1} F_{2s-t} \Big ) \nonumber \\
&\qquad\qquad + \frac{1}{n+1} \frac{F_{r+s}}{F_r} (-1)^{s+t+1} F_{s-t}
\end{align}
and 
\begin{align}
&F_{r+s}^{-n} \sum_{k = 0}^n \binom{n}{k} \frac{1}{k+2} (-1)^{r(n-k)} F_{r}^{k} F_s^{n-k} L_{sk-r(n-k)+t}  \nonumber \\
&= \sum_{k=0}^n \binom{n}{k} \frac{1}{(k+1)(k+2)}(-1)^{k} \Big (\frac{F_r}{F_{r+s}}\Big )^{k} L_{sk+t} \nonumber \\
&=\frac{1}{(n+1)(n+2)}\Big ( \Big (\frac{F_s}{F_r}\Big )^2 \Big (\frac{F_s}{F_{r+s}}\Big )^n (-1)^{t} L_{2s+r(n+2)-t}
- \Big (\frac{F_{r+s}}{F_r}\Big )^2 (-1)^{t} L_{2s-t} \Big ) \nonumber \\
&\qquad\qquad + \frac{1}{n+1} \frac{F_{r+s}}{F_r} (-1)^{s+t} L_{s-t}.
\end{align}
In particular,
\begin{equation}
\sum_{k = 0}^n \binom{n}{k} \frac{1}{k+2} (-1)^{k+1} F_{n-2k} = \sum_{k=0}^n \binom{n}{k} \frac{(-1)^{k} F_{k}}{(k+1)(k+2)} 
=\frac{1 - F_{n+4}}{(n+1)(n+2)} + \frac{1}{n+1}
\end{equation}
and
\begin{equation}
\sum_{k = 0}^n \binom{n}{k} \frac{1}{k+2} (-1)^{k} L_{n-2k} = \sum_{k=0}^n \binom{n}{k} \frac{(-1)^{k} L_{k}}{(k+1)(k+2)} 
=\frac{L_{n+4} - 3}{(n+1)(n+2)} - \frac{1}{n+1}.
\end{equation}
\end{theorem}

\begin{theorem}\label{thm_sumrel2}
If $s$ is an even integer and $t$ is any integer, then
\begin{align}
&\sum_{k = 0}^n \binom{n}{k} \frac{1}{k+2} (-1)^k L_s^{n-k} F_{s(n+k)+t} = \sum_{k=0}^n \binom{n}{k} \frac{F_{2sk+t}}{(k+1)(k+2)} \nonumber \\
&= \frac{1}{(n+1)(n+2)}\Big (L_s^{n+2} F_{s(n-2)+t}+(-1)^t F_{4s-t}\Big ) + \frac{(-1)^t}{n+1} F_{2s-t}
\end{align}
and 
\begin{align}
&\sum_{k = 0}^n \binom{n}{k} \frac{1}{k+2} (-1)^k L_s^{n-k} L_{s(n+k)+t} = \sum_{k=0}^n \binom{n}{k} \frac{L_{2sk+t}}{(k+1)(k+2)} \nonumber \\
&= \frac{1}{(n+1)(n+2)}\Big (L_s^{n+2} L_{s(n-2)+t}-(-1)^t L_{4s-t}\Big ) - \frac{(-1)^t}{n+1} L_{2s-t}.
\end{align}
In particular,
\begin{align}
\sum_{k = 0}^n \binom{n}{k} \frac{1}{k+2} (-1)^k 3^{n-k} F_{2(n+k)} &= \sum_{k=0}^n \binom{n}{k} \frac{F_{4k}}{(k+1)(k+2)} \nonumber \\
&= \frac{1}{(n+1)(n+2)}\Big (3^{n+2} F_{2(n-2)} + 21 \Big ) + \frac{3}{n+1}
\end{align}
and 
\begin{align}
\sum_{k = 0}^n \binom{n}{k} \frac{1}{k+2} (-1)^k 3^{n-k} L_{2(n+k)} &= \sum_{k=0}^n \binom{n}{k} \frac{L_{4k}}{(k+1)(k+2)} \nonumber \\
&= \frac{1}{(n+1)(n+2)}\Big (3^{n+2} L_{2(n-2)} - 47 \Big ) - \frac{7}{n+1}.
\end{align}
\end{theorem}

\section{Conclusion}

Motivated by the author's recent problem proposal closed forms for three new classes of binomial sums 
with Fibonacci and Lucas numbers were derived. In addition, a few sum relations connected with the subject were discussed. 
Extensions of the results presented this note to gibonacci or even to Horadam sequences should be possible with little effort. 
This exercise is left to the interested readers.


\begin{thebibliography}{99}

\bibitem{Adegoke1}
Adegoke K. \emph{Weighted sums of some second-order sequences,} Fibonacci Quart. 56 (3), 2018, 252--262. 

\bibitem{Adegoke2}
K. Adegoke, A. Olatinwo, S. Ghosh, \emph{Cubic binomial Fibonacci sums}, Electron. J. Math. 2, 2021, 44--51.

\bibitem{Adegoke3}
K. Adegoke, R. Frontczak and T. Goy, \emph{Binomial Fibonacci sums from Chebyshev polynomials}, 
J. Integer Seq., 26 (9), Article 23.9.6, 2023.

\bibitem{Adegoke4}
K. Adegoke, R. Frontczak and T. Goy, \emph{Binomial sum relations involving Fibonacci and Lucas numbers},
AppliedMath 1, 2023, 1--31.

\bibitem{Bai}
M. Bai, W. Chu and D. Guo, \emph{Reciprocal formulae among Pell and Lucas polynomials}, Mathematics 10, 2022, 2691.

\bibitem{Carlitz}
L. Carlitz, \emph{Some classes of Fibonacci sums,} Fibonacci Quart. 16 (5), 1978, 411--425.

\bibitem{CarFer}
L. Carlitz and H. H. Ferns, \emph{Some Fibonacci and Lucas identities,} Fibonacci Quart. 8 (1), 1970, 61--73.

\bibitem{Dattoli}
G. Dattoli, S. Licciardi and R. M. Pidatella, \emph{Inverse derivative operator and umbral methods for
the harmonic numbers and telescopic series study}, Symmetry 13, 2021, 781.

\bibitem{Frontczak}
R. Frontczak, \emph{Advanced Problem H-882}, Fibonacci Quart. 59 (3), 2021, 281.

\bibitem{Hoggatt}
V. E. Hoggatt, Jr. and M. Bicknell, \emph{Some new Fibonacci identities,} Fibonacci Quart. 2 (1), 1964, 121--133.

\bibitem{Hoggatt2}
V. E. Hoggatt, Jr. J. W. Phillips and H. T. Leonard, Jr., \emph{Twenty-four master identities,} Fibonacci Quart. 9 (1), 1971, 1--17.

\bibitem{Jen}
D. Jennings, \emph{Some polynomial identities for the Fibonacci and Lucas numbers}, Fibonacci Quart. 31, 1993, 134--137.

\bibitem{Kilic1}
E. Kilic and E. J. Ionascu, \emph{Certain binomial sums with recursive coefficients}, Fibonacci Quart. 48 (2), 2010, 161--167.

\bibitem{Kilic2}
E. Kilic, N. {\"O}m{\"u}r and Y. T. Ulutas, \emph{Binomial sums whose coefficients are products of terms of binary sequences,} 
Util. Math. 84, 2011, 45--52.

\bibitem{Koshy} 
T. Koshy, \emph{Fibonacci and Lucas Numbers with Applications}, Wiley-Interscience, 2001.

\bibitem{Layman}
J. W. Layman, \emph{Certain general binomial-Fibonacci sums}, Fibonacci Quart. 15 (3), 1977, 362--366.

\bibitem{Long} 
C. T. Long, \emph{Some binomial Fibonacci identities}, Applications of Fibonacci Numbers, Vol. 3, Dordrecht: Kluwer, 1990, pp. 241--254.

\bibitem{Vajda} 
S. Vajda, \emph{Fibonacci and Lucas Numbers, and the Golden Section: Theory and Applications}, Dover Press, 2008.

\bibitem{Ventas}
A. Ventas, \emph{Solution to Advanced Problem H-882}, Fibonacci Quart. 61 (1), 2023, 95--96.

\bibitem{Zeitlin}
D. Zeitlin, \emph{General identities for recurrent sequences of order two}, Fibonacci Quart. 9, 1971, 357--388.

\bibitem{OEIS}
N. J. A. Sloane, \textit{The On-Line Encyclopedia of Integer Sequences}, https://oeis.org.

\end{thebibliography}
\end{document}